\documentclass{aims}
\usepackage{amsmath}
  \usepackage{paralist}
  \usepackage{graphics} 
  \usepackage{epsfig} 
\usepackage{graphicx}  \usepackage{epstopdf}
 \usepackage[colorlinks=true]{hyperref}
\hypersetup{urlcolor=blue, citecolor=red}

\def\currentvolume{}
 \def\currentissue{}
  \def\currentyear{}
   \def\currentmonth{}
    \def\ppages{}
     \def\DOI{}

\newtheorem{theorem}{Theorem}[section]
\newtheorem{corollary}{Corollary}

\newtheorem{lemma}[theorem]{Lemma}
\newtheorem{proposition}{Proposition}

\theoremstyle{definition}

\newtheorem{remark}{Remark}

\title[Product of statistical manifolds with a non-diagonal metric] 
      {Product of statistical manifolds with a non-diagonal metric}

\author[Rafik Nasri]{}


  \email{rmag1math@yahoo.fr}
  \email{ddjebbouri20@gmail.com}

\begin{document}
\maketitle

\maketitle


\centerline{\scshape Djebbouri Djelloul and Rafik Nasri}
\medskip
{\footnotesize
 \centerline{Laboratory of Geometry, Analysis, Control and Applications}
   \centerline{Universit\'e de Sa\"{i}da}
   \centerline{BP138, En-Nasr, 20000 Sa\"ida, Algeria}
}

\bigskip

 \centerline{(Communicated by the associate editor name)}

\begin{abstract}
In this paper, we generalize the dualistic structures on warped product
manifolds to the dualistic structures on generalized warped product manifolds.
we develop an expression of curvature for the connection of the generalized
warped product in relation to those corresponding analogues of its base and fiber
and warping functions. we show that  the dualistic structures on the base $M_{_1}$
and the fiber $M_{_2}$ induces a dualistic structure on the generalized warped product
$M_1\times M_2$ and conversely, moreover, $(M_{_1}\times M_{_1},G_{_{f_{_1}f_{_2}}})$
or $(M_{_1}\times M_{_1},\tilde{g}_{_{f_{_1}f_{_2}}})$ is statistical manifold if and
only if $(M_{_1},g_{_1})$ and $(M_{_1},g_{_1})$ are. Finally, Some interesting consequences
are also given.
\end{abstract}


\section{Introduction}
The warped product provides a way to construct new pseudo-rieman
nian manifolds from the given ones, see \cite{B. O'Neill},\cite{bishop} and
\cite{Beem2}. This construction has useful applications in general
relativity, in the study of cosmological models and black holes. It
generalizes the direct product in the class of pseudo-Riemannian
manifolds and it is defined as follows. Let $(M_1,g_1)$ and
$(M_2,g_2)$ be two pseudo-Riemannian manifolds and let
$f_1:M_1\longrightarrow\mathbb{R}^*$ be a positive smooth function on
$M_1$, the warped product of $(M_1,g_1)$ and
$(M_2,g_2)$ is the product manifold $M_1\times M_2$ equipped
with the metric tensor $g_{f_1}:=\pi_1^*g_1
+(f\circ \pi_1)^2\pi_2^*g_2$, where $\pi_1$ and
$\pi_2$ are the projections of $M_1\times M_2$ onto $M_1$ and
$M_2$ respectively. The manifold $M_1$ is called the base of
$(M_1\times M_2,g_{f_1})$ and $M_2$ is called the fiber. The function $f_1$
is called the warping function.  \\
The doubly warped product is  construction in the class of pseudo-Riemannian
manifolds generalizing the warped product and the direct product, it is obtained by homothetically
distorting the geometry of each base $M_{_1}\times\{q\}$ and each fiber $\{p\}\times M_{_2}$ to get a new
"doubly warped" metric tensor on the product manifold and defined as follows.
For $i\in \{1,2\}$, let $M_i$ be a pseudo-Riemannian manifold equipped with metric
$g_i$, and $f_{_i}: M_i\rightarrow \mathbb{R}^*$ be a positive smooth function on $M_i$. The well-know notion of doubly warped product
manifold $M_{_1}\times _{_{f_1f_2}}M_{_2}$ is defined as the product manifold $M=M_{_1}\times M_{_2}$
equipped with pseudo-Riemannian metric which is denoted by $g_{_{f_1f_2}}$, given by
$$
g_{_{f_1f_2}}=(f_2\circ\pi_2)^2\pi_1^* g_{_{_1}}+(f_1\circ\pi_1)^2\pi_2^*g_{_{_2}}.
$$
The generalized warped product is defined as follows. let $c$ be an arbitrary real number and
let $g_i$, $(i=1,2)$ be a Riemannian metric tensors on $M_i$. Given a smooth positive function
$f_i$ on $M_i$, the generalized warped product of $(M_1,g_1)$ and $(M_2,g_2)$ is the product
manifold $M_1\times M_2$ equipped with the metric tensor $G_{f_1f_2}$ (see \cite{Nass2}),
explicitly, given by
$$
\begin{array}{rl}
 G_{_{f_1,f_2}}(X,Y)&=(f_2^v)^2g_{_{_1}}^{\pi_1}(d\pi_1(X),d\pi_1(Y))
 +(f_1^h)^2g_{_{_2}}^{\pi_2}(d\pi_2(X),d\pi_2(Y)) \\
 &\\
&+cf_1^hf_2^v\left(X(f_1^h)Y(f_2^v)+X(f_2^v)Y(f_1^h))\right).
\end{array}
$$
For all $X,Y\in\Gamma(TM_1\times M_2)$. When the warping functions $f_1=1$ or $f_2=1$ or $c=0$
we obtain a warped product or direct product. \\

Dualistic structures are closely related to statistical mathematics. They consist of pairs of affine
connections on statistical manifolds, compatible with a pseudo-Riemanniann metric \cite{Amari Dif-geom}.
 Their importance in statistical physics was underlined by many authors: \cite{Furuhata},\cite{Amari Nagaoka} etc.\\
 Let $M$ be a pseudo-Riemannian manifold equipped with a pseudo-Riemannian metric $g$ and let $\nabla$, $\nabla^{^{*}}$
 be the affine connections on $M$. We say that a pair of affine connections  $\nabla$ and $\nabla^{^{*}}$
 are compatible (or conjugate ) with respect to $g$ if
 \begin{equation}\label{dual}
    X(g(Y,Z))=g(\nabla_XY,Z)+g(Y,\nabla^{^{*}}_XZ) ~~~\text{for~all} ~X,Y,Z\in \Gamma(TM),
 \end{equation}
 where $\Gamma(TM)$ is the set of all tangent vector fields on $M$. Then the triplet $(g,\nabla,\nabla^{^{*}})$
 is called the dualistic structure on $M$. \\
 We note that the notion of "conjugate connection " has been attributed to A.P. Norden in affine differential
 geometry literarture (Simon, 2000) and was independently introduced by (Nagaoka and Amari, 1982) in information
 geometry, where it was called " dual connection" (Lauritzen, 1987).
 The triplet $(M, \nabla,g)$ is called a statistical manifold if it admits another torsion-free
 connection $\nabla^{^{*}}$ satisfying the equation (\ref{dual}). We call $\nabla$ and $\nabla^{^{*}}$ duals
 of each other with respect to $g$.\\

In the notions of terms on statistical manifolds, for a torsion-free affine
connection $\nabla$ and a pseudo-Riemannian metric $g$ on a manifold $M$, the
triple $(M, \nabla, g)$ is called a statistical manifold if $\nabla g$ is symmetric.
If the curvature tensor $R$ of $\nabla$ vanishes, $(M,\nabla,g)$ is said to be flat.\\

This paper extends the study of dualistic structures on warped product manifolds,
\cite{TOdjihounde}, to dualistic structures on generalized warped products in
pseudo-Riemannian manifolds. We develop an expression of curvature for the connection of
the generalized warped product in relation to those corresponding analogues of
its base and fiber and warping functions.\\
The paper is organized as follows. In section 2, we collect the basic material about
Levi-Civita connection, the notion of conjugate, horizontal and vertical lifts and the
 generalized warped products.\\
In section 3, we show that the projection of a dualistic structure defined on a generalized
warped product space $(M_1\times M_2, G_{f_1f_2})$ induces dualistic structures on the base
$(M_1,g_1)$ and the fiber $(M_2,g_2)$. Conversely, there exists a dualistic structure on the
generalized warped product space induced by its base and fiber.\\
In section 4,  we show that the projection of a dualistic structure defined on a generalized
warped product space $(M_1\times M_2, \tilde{g}_{f_1f_2})$ induces dualistic structures on the base
$(M_1,g_1)$ and the fiber $(M_2,g_2)$. Conversely, there exists a dualistic structure on the
generalized warped product space induced by its base and fiber and finally,
Some interesting consequences are also given.

\section{Preliminaries}
\subsection{Statistical manifolds}
We recall  some standard facts about Levi-Civita connections and the dual statistical manifold.
Many fundamental definitions and results about dualistic structure can
be found in Amari's monograph (\cite{Amari Dif-geom},\cite{Amari Nagaoka}).

Let $(M,g)$ be a pseudo-Riemannian manifold. The metric
$g$ defines the musical isomorphisms
$$
\begin{array}{cccc}
\sharp_g: & T^*M&  \rightarrow  & TM \\
& \alpha & \mapsto & \sharp_g(\alpha)
\end{array}
$$
such that $g(\sharp_g(\alpha),Y)=\alpha(Y)$, and its inverse $\flat_g$.
We can thus define the cometric $\widetilde{g}$ of the metric $g$ by :
\begin{equation}\label{musical}
\widetilde{g}(\alpha,\beta)=g(\sharp_g(\alpha),\sharp_g(\beta)).
\end{equation}
A fundamental theorem of pseudo-Riemannian geometry states that given a pseudo-Riemannian metric $g$ on the tangent
bundle $TM$, there is a unique connection (among the class of torsion-free connection) that "preserves" the
metric; as long as the following condition is satisfied:
\begin{equation}\label{compatible}
    X(g(Y,Z))=g(\hat{\nabla}_XY,Z)+g(Y,\hat{\nabla}_XZ) ~~~for ~X,Y,Z\in \Gamma(TM)
\end{equation}
Such a connection, denoted as $\hat\nabla$, is known as the Levi-Civita connection. Its component
forms, called Christoffel symbols, are determined by the components of pseudo-metric tensor as
("Christoffel symbols of the second Kink ")
$$
\hat{\Gamma}_{ij}^k=\sum_l\frac{1}{2}g^{kl}(\frac{\partial g_{il}}{\partial x^j}
+\frac{\partial g_{jl}}{\partial x^i}-\frac{\partial g_{ij}}{\partial x^l})
$$
and ("Christoffel symbols of the first Kink")
$$
\hat{\Gamma}_{ij,k}=\frac{1}{2}(\frac{\partial g_{ik}}{\partial x^j}
+\frac{\partial g_{jk}}{\partial x^i}-\frac{\partial g_{ij}}{\partial x^k}).
$$
The Levi-Civita connection is compatible with the pseudo metric, in the sense that it treats
tangent vectors of the shortest curves on a manifold as being parallel.\\

It turns out that one can define a kind of "Compatibility" relation more generally than expressed by
the (\ref{compatible}), by introducing the notion of "Conjugate" (denoted by *) between two affine
 connections.\\

Let $(M,g)$ be a pseudo-Riemannian manifold and let $\nabla$, $\nabla^{^{*}}$ be an affine connections
on $M$. A connection $\nabla^{^{*}}$ is said to be "conjugate" to $\nabla$ with respect to $g$ if
\begin{equation}
    X(g(Y,Z))=g(\nabla_XY,Z)+g(Y,\nabla^{^{*}}_XZ) ~~~for ~X,Y,Z\in \Gamma(TM)
 \end{equation}
Clearly,
$$
(\nabla^{^{*}})^{^{*}}=\nabla.
$$
Otherwise, $\hat{\nabla}$, which satisfies the (\ref{compatible}), is special in the sense that it is self-conjugate
$$
(\hat{\nabla})^{^{*}}=\hat{\nabla}.
$$
Because pseudo-metric tensor $g$ provides a one-to-one mapping between vectors in the tangent space
and co-vectors in the cotangent space, the equation (\ref{dual}) can also be seen as characterizing how co-vector
fields are to be parallel-transported in order to preserve their dual pairing $<.,.>$ with vector fields.\\
Writing out the equation \ref{dual} explicitly,
\begin{equation}
    \frac{\partial g_{ij}}{\partial x^k}=\Gamma_{ki,j}+\Gamma_{kj,i}^*,
\end{equation}
where
$$
\nabla^{^{*}}_{\partial_i}\partial_j=\sum_l\Gamma_{ij}^{*l}\partial_l
$$
so that
$$
\Gamma_{kj,i}^*=g(\nabla^{^{*}}_{\partial_j}\partial_k,\partial_i)=\sum_lg_{il}\Gamma_{kj}^{*l}.
$$
In the following, a manifold $M$ with a pseudo-metric $g$ and a pair of conjugate connections $\nabla, \nabla^{^{*}}$
with respect to $g$ is called a " pseudo-Riemannian manifold with dualistic structure " and denoted by $(M,g,\nabla,\nabla^{^{*}})$.\\
Obviously, $\nabla$ and $\nabla^{^{*}}$ (or equivalently, $\Gamma$ and $\Gamma^*$) satisfy the relation
$$
\hat{\nabla}=\frac{1}{2}(\nabla+ \nabla^{^{*}}) ~~~~(\text{or~ equivalently}, \hat{\Gamma}=\frac{1}{2}(\Gamma+\Gamma^*)).
$$
Thus an affine connection $\nabla$ on $(M,g)$ is metric if and only if $\nabla^{^{*}}=\nabla$ ( that it is self-conjugate).   \\
For a torsion-free affine connection $\nabla$ and a pseudo-Riemannian metric $g$ on a manifold $M$, the triplet
$(M, \nabla, g)$ is called a statistical manifold if $\nabla g$ is symmetric. If the curvature tensor $\mathcal{R}$
of $\nabla$ vanishes, $(M, \nabla, g)$ is said to be flat.\\
For a statistical manifold $(M, \nabla, g)$, the conjugate connection $\nabla^{^{*}}$ with respect to $g$ is torsion-free
and $\nabla^{^{*}} g$ symmetric. Then the triplet  $(M, \nabla^{^{*}}, g)$ is called the dual statistical manifold of  $(M, \nabla, g)$
and  $( \nabla, \nabla^{^{*}}, g)$ the dualistic structure on $M$. The curvature tensor of $\nabla$ vanishes if and
only if that of $\nabla^{^{*}}$ does and in such a case, $( \nabla, \nabla^{^{*}}, g)$
is called the dually flat structure \cite{Amari Nagaoka}.\\
More generally, in information geometry, a one-parameter family of affine connections $\nabla^{(\lambda)}$
indexed by $\lambda$ $(\lambda\in \mathbb{R})$, called $\lambda-$ connections, is introduced by Amari
and Nagaoka in (\cite{Amari Dif-geom},\cite{Amari Nagaoka}).\\
\begin{equation}
    \nabla^{(\lambda)}=\frac{1+\lambda}{2}\nabla+\frac{1-\lambda}{2}\nabla^{^{*}} ~~~
    (\text{or~ equivalently},\Gamma^{(\lambda)}=\frac{1+\lambda}{2}\Gamma+\frac{1-\lambda}{2}\Gamma^*).
\end{equation}
Obviously, $ \nabla^{(0)}=\hat{\nabla}$.\\
It can be shown that for a pair of conjugate connections $\nabla, \nabla^{^{*}}$, their
curvature tensors $R$, $\mathcal{R}^{^{*}}$  satisfy
\begin{equation}\label{R et R*}
    g(\mathcal{R}(X,Y)Z,W)+g(Z,\mathcal{R}^{^{*}}(X,Y)W)=0,
\end{equation}
and more generally
\begin{equation}
   g(\mathcal{R}^{(\lambda)}(X,Y)Z,W)+g(Z,\mathcal{R}^{*(\lambda)}(X,Y)W)=0.
\end{equation}
If the curvature tensor $\mathcal{R}$ of $\nabla$ vanishes, $\nabla$ is said
to be flat.
\\So, $\nabla$ is flat if and only if $\nabla^*$ is flat. In this case,
$(M,g,\nabla,\nabla^{^{*}})$ is said to be dually falt.\\
When $\nabla,\nabla^{^{*}}$ is dually flat, then $\nabla^{(\lambda)}$ is called
$\lambda$-transitively flat \cite{Uohashi K}. In such case, $(M,g,\nabla^{(\lambda)}
,\nabla^{*(\lambda)})$ is called an "$\lambda$-Hessian manifold", or a manifold
with $\lambda$-Hessian structure.\\

\subsection{Horizontal and vertical lifts}
Throughout this paper $M_{1}$ and $M_{2}$ will be respectively
$m_{1}$ and $m_{2}$ dimensional manifolds, $M_1\times M_2$ the
product manifold with the natural product coordinate system and
$\pi_1:M_{1}\times M_{2}\rightarrow M_{1}$ and $\pi_2
:M_{1}\times M_{2}\rightarrow M_{2}$ the usual projection maps.

We recall briefly how the calculus on the product manifold $M_1
\times M_2$ derives from that of $M_1$ and $M_2$ separately. For
details see \cite{B. O'Neill}.

Let $\varphi _{1}$ in $C^{\infty }(M_{1})$. The horizontal lift of
$\varphi_{1}$ to $M_{1}\times M_{2}$ is $\varphi_{1}^{h}=\varphi
_{1}\circ \pi_1$. One can define the horizontal lifts of
tangent vectors as follows. Let $p_1\in M_1$ and let $X_{p_1}\in
T_{p_1}M_{1}$. For any $p_2\in M_{2}$ the horizontal lift
 of $X_{p_1}$ to $(p_1,p_2)$ is the unique tangent vector $X_{(p_1,p_2)}^{h}$
in $T_{(p_1,p_2)}(M_{1}\times M_2)$ such that
$ d_{(p_1,p_2)}\pi_1(X_{(p_1,p_2)}^{h})=X_{p_1}$ and $d_{(p_1,p_2)}\pi_2(X_{(p_1,p_2)}^{h})=0.$\\
We can also define the horizontal lifts of vector fields as follows.
Let $X_1\in \Gamma (TM_{1})$. The horizontal lift of $X_1$ to
$M_{1}\times M_{2}$ is the vector field $X_1^{h}\in \Gamma
(T(M_{1}\times M_{2}))$ whose value at each $(p_1,p_2)$ is the horizontal
lift of the tangent vector $(X_1){p_1}$ to $(p_1,p_2)$. For $(p_1,p_2)\in M_1\times M_2$,
we will denote the set of the horizontal lifts to $(p_1,p_2)$ of all the tangent
vectors of $M_{1}$ at $p_1$ by $L(p_1,p_2)(M_{1})$. We will denote the set of the
horizontal lifts of all vector fields on $M_{1}$ by $\mathfrak{L}(M_{1})$.

The vertical lift $\varphi_2^v$ of a function $\varphi_2\in
C^{\infty}(M_2)$ to $M_1\times M_2$ and the vertical lift $X_2^v$ of
a vector field $X_2\in \Gamma (TM_{2})$ to $M_1\times M_2$ are
defined in the same way using the projection $\pi_2$. Note that
the spaces $\mathfrak{L}(M_{1})$ of the horizontal lifts and
$\mathfrak{L}(M_{2})$ of the vertical lifts are vector subspaces of
$\Gamma (T(M_{1}\times M_{2}))$ but neither is invariant under
multiplication by arbitrary functions $\varphi \in C^{\infty
}(M_{1}\times M_{2})$.\\

Observe that if $\{\frac{\partial}{\partial x_1},\ldots,\frac{\partial}
{\partial x_{m_1}}\}$ is the local basis of the vector fields (resp.
$\{dx_1,\ldots,dx_{m_1}\}$ is the local basis of $1$-forms ) relative to a chart
$(U,\Phi)$ of $M_1$ and $\{\frac{\partial}{\partial y_1},\ldots,\frac{\partial}
{\partial y_{m_2}}\}$ is the local basis of the vector fields (resp. $\{dy_1,
\ldots,dy_{m_2}\}$ the local basis of the $1$-forms) relative to a chart $(V,\Psi)$
of $M_2$, then $\{(\frac{\partial}{\partial x_1})^h,\ldots,(\frac{\partial}
{\partial x_{m_1}})^h,(\frac{\partial}{\partial y_1})^v,\\
\ldots,(\frac{\partial}{\partial y_{m_2}})^v \}$ is the local
basis of the vector fields (resp. $\{(dx_1)^h,\ldots,(dx_{m_1})^h,
(dy_1)^v,\\
\ldots,(dy_{m_2})^v\}$ is the local basis of the $1$-forms)  relative to the chart $(U\times
V,\Phi\times \Psi)$ of $M_1\times M_2$.\\

The following lemma will be useful later for our computations.
\begin{lemma} \label{lift} $\;$
\begin{enumerate}
\item Let $\varphi_i\in C^\infty(M_i)$, $X_i,Y_i\in \Gamma (TM_{i})$
and $\alpha _{i}\in \Gamma (T^* M_{i})$, $i=1,2$. Let
$\varphi=\varphi_1^h+\varphi_2^v$, $X=X_{1}^{h}+X_{2}^{v}$ and
$\alpha ,\beta \in \Gamma (T^*(M_{1}\times M_{2}))$. Then
\begin{enumerate}
\item[i/] For all $(i,I)\in \{(1,h),(2,v)\}$, we have
$$
X_i^I(\varphi)=X_i(\varphi_i)^I,\quad [X,Y_i^I]=[X_i,Y_i]^I \quad
\textrm{ and } \quad  \alpha _{i}^{I}(X)=\alpha_{i}(X_{i})^{I}.
$$
\item[ii/] If for all $(i,I)\in \{(1,h),(2,v)\}$ we have $\alpha
(X_{i}^{I})=\beta (X_{i}^{I})$, then $\alpha =\beta$.
\end{enumerate}
\item Let $\omega_i$ and $\eta_i$ be $r$-forms on $M_i$, $i=1,2$.
Let $\omega=\omega_1^h+\omega_2^v$ and $\eta=\eta_1^h+\eta_2^v$. We
have
$$
d\omega=(d\omega_1)^h+(d\omega_2)^v \quad \textrm{ and } \quad
\omega \wedge \eta=(\omega_1\wedge \eta_1)^h+(\omega_2\wedge
\eta_2)^v.
$$
\end{enumerate}
\end{lemma}
\begin{proof}
See \cite{Nas}.
\end{proof}
\begin{remark}\label{rem dpi}
    Let $X$ be a vector field on $M_1\times M_2$, such that $d\pi_1(X)=\varphi(X_1\circ\pi_1)$
and $d\pi_2(X)=\phi(X_2\circ\pi_2)$, then $X=\varphi X_1^h+\phi X_2^v$.
\end{remark}

\subsection{The generalized warped product.}
let $\psi: M\rightarrow N$ be a smooth map between smooth manifolds and let $g$ be a metric on
 $k$-vector bundle $(F,P_F)$ over $N$. The metric
$g^{\psi}: \Gamma(\psi^{-1}F)\times \Gamma(\psi^{-1}F)\rightarrow C^{\infty}(M)$ on the pull-back
$(\psi^{-1}F,P_{\psi^{-1}F})$ over $M$  is defined by
$$
g^{\psi}(U,V)(p)=g_{\psi(p)}(U_p,V_p),~ ~\forall~ U,V\in \Gamma(\psi^{-1}F),~ p\in M.
$$
Given a linear connection $\nabla^N$ on $k$-vector bundle $(F,P_F)$ over $N$, the pull-back
connection $\nabla{\hskip-0.3cm^{^{^{\psi}}}}$ is the unique linear connection on the pull-back
$(\psi^{-1}F,P_{\psi^{-1}F})$ over $M$ such that
\begin{equation}
  \nabla{\hskip-0.3cm^{^{^{\psi}}}}_{X} \big(W\circ\psi\big)=\nabla^N{\hskip -0.4cm_{_{_{d\psi(X)}}}}\hskip -0.3cm W,
   \hskip 0.4cm \forall W\in\Gamma(F),~ \forall X\in\Gamma(TM).
\end{equation}
Further, let $U\in\psi^{-1}F$ and let $p\in M$, $X\in\Gamma(TM)$. Then
\begin{equation}
   (\nabla{\hskip-0.3cm^{^{^{\psi}}}}_{X} U)(p)=(\nabla^N{\hskip -0.4cm_{_{_{d{\!_{_{_p}}}
   \!\!\!\psi(X_{_p}\!)}}}}\hskip -0.4cm\widetilde{U})(\psi(p)),
\end{equation}
where $\widetilde{U}\in\Gamma(F)$ with $\widetilde{U}\circ\psi=U$.\\
Now, let $\pi_i$, i=1,2, be the usual projection of $M_1\times M_2$ onto $M_i$,
given a linear connection $\nabla{\hskip-0.2cm^{^{^{i}}}}$ on vector bundle $TM_i$, the pull-back
connection $\nabla{\hskip-0.4cm^{^{^{\pi_i}}}}$ is the unique linear connection on the pull-back
$M_1\times M_2\rightarrow\pi_i^{-1}(TM_i)$ such that for each $Y_i\in\Gamma(TM_i)$,
 $X\in\Gamma(TM_1\times M_2)$
\begin{equation}
  \nabla{\hskip-0.3cm^{^{^{\pi_i}}}}_{X} \big(Y_i\circ\pi_i\big)=\nabla{\hskip-0.2cm^{^{^{i}}}}
  {\hskip -0.05cm_{_{_{_{d\pi_i(X)}}}}}\hskip -0.5cmY_i.
\end{equation}
Further, let $(p_1,p_2)\in M_1\times M_2$, $U\in\pi_i^{-1}(TM)$ and $X\in\Gamma(TM_1\times M_2)$. Then
\begin{equation}
   (\nabla{\hskip-0.3cm^{^{^{\pi_i}}}}_{X}U)(p_1,p_2)=\big(\nabla{\hskip-0.2cm^{^{^{i}}}}
   _{d{\hskip -0.2cm_{_{_{(p_1,p_2)}}}}\hskip -0.6cm\pi_i(X_{(p_1,p_2)})}{\widetilde{U}}\big)(p_i),
\end{equation}

Now, let $c$ be an arbitrary real number and let $g_i$, $(i=1,2)$ be a Riemannian metric
tensors on $M_i$. Given a smooth positive function $f_i$ on $M_i$, the generalized
warped product of $(M_1,g_1)$ and $(M_2,g_2)$ is the product manifold $M_1\times M_2$ equipped
with the metric tensor (see \cite{Nass2})
$$
G_{_{f_1,f_2}}=(f_2^v)^2\pi_1^*g_{_{_1}}+(f_1^h)^2\pi_2^*g_{_{_2}} +cf_1^hf_2^v df_1^h\odot df_2^v,
$$
Where $\pi_i$, $(i=1,2)$ is the projection of $M_{_1}\times M_{_2}$ onto $M_{_i}$ and
$$
df_1^h\odot df_2^v = df_1^h\otimes df_2^v + df_2^v\otimes df_1^h.
$$
For all $X,Y\in\Gamma(TM_1\times M_2)$, we have
$$
\begin{array}{rl}
 G_{_{f_1,f_2}}(X,Y)&=(f_2^v)^2g_{_{_1}}^{\pi_1}(d\pi_1(X),d\pi_1(Y))
 +(f_1^h)^2g_{_{_2}}^{\pi_2}(d\pi_2(X),d\pi_2(Y)) \\
 &\\
&+cf_1^hf_2^v\left(X(f_1^h)Y(f_2^v)+X(f_2^v)Y(f_1^h))\right).
\end{array}
$$
It is the unique tensor fields such that for any $X_i,Y_i\in\Gamma(TM_i)$, $(i=1,2)$
\begin{equation} \label{equivalent generalized warped}
 \tilde{g}_{_{f_1f_2}}(X_i^I,Y_k^K)= \left\{
  \begin{array}{ccl}
   (f_{3-i}^{J})^2g_i(X_i,Y_i)^I,& &\text{if} ~(i,I)=(k,K) \\
&&\\
    cf_i^If_k^KX_i(f_i)^IY_k(f_k)^K,& & \text{otherwise}
  \end{array}
\right.
\end{equation}
If either $f_1\equiv 1$ or $f_2\equiv 1$ but not both, then we obtain a singly warped
product. If both $f_1\equiv 1$ and $f_2\equiv 1$, then we have a product manifold.
If neither $f_1$ nor $f_2$ is constant and $c=0$, then we have a nontrivial doubly warped product.
If neither $f_1$ nor $f_2$ is constant and $c\neq 0$, then we have a nontrivial generalized warped product.

Now, Let us assume that $(M_i,g_i)$, $(i=1,2)$ is a smooth connected Riemannian manifold.
The following proposition provides a necessary and sufficient condition for a symmetric tensor field
 $G_{f_1,f_2}$ of type $(0,2)$ of two Riemannian metrics to be a Riemannian metric.

 \begin{proposition}\label{condition generalized warped}\cite{Nass2}
Let $(M_i,g_i)$, $(i=1,2)$ be a Riemannian manifold
and let $f_i$ be a positive smooth function on $M_i$ and $c$ be an arbitrary real number.
Then the symmetric tensor field
$G_{f_1f_2}$is Riemannian metric on $M_1\times M_2$ if and only if
\begin{equation}\label{condition of metric}
   0 \leq c^2g_1(gradf_1,gradf_1)^h g_2(gradf_2,gradf_2)^v<1.
\end{equation}
\end{proposition}
\begin{corollary}\cite{Nass2}
If the symmetric tensor field $G_{f_1,f_2}$ of type $(0,2)$ on $M_1\times M_2$ is degenerate,
then for any  $i\in\{1,2\}$, $g_i(gradf_i,gradf_i)$ is positive constant $k_i$ with
$$
k_i=\frac{1}{c^2k_{(3-i)}}.
$$
\end{corollary}
In all what follows, we suppose that $f_1$ and $f_2$ satisfies the inequality (\ref{condition of metric}).
\begin{lemma}\label{calculate of X on G}\cite{Nass2}
 Let $X$ be an arbitrary vector field of $M_1\times M_2$, if there exist $\varphi_i,\psi_i\in C^{\infty}(M_i)$
and $X_i,Y_i\in \Gamma(TM_i)$, $(i=1,2)$ such that
$$
  \left\{
     \begin{array}{lll}
       G_{f_1f_2}(X,Z_1^h)= G_{f_1f_2}(\varphi_2^vX_1^h+\varphi_1^hX_2^v,Z_1^h), & \\
&&\forall ~Z_i\in \Gamma(TM_i),\\
       G_{f_1f_2}(X,Z_2^v)= h^hG_{f_1f_2}(\psi_2^vY_1^h+\psi_1^hY_2^v,Z_2^v). &
     \end{array}
   \right.
$$
Then we have,
\begin{equation}
\begin{array}{ccl}
  X& = &\varphi_2^vX_1^h+\psi_1^hY_2^v+cf_1^hf_2^v\left\{\psi_{2}^vY_1(f_1)^h
\!-\!\varphi_{2}^vX_1(f_1)^h\right\}grad(f_{2}^v) \\
&&\\
 &-&cf_1^hf_2^v\left\{\psi_{1}^hY_2(f_2)^v
\!-\!\varphi_{1}^hX_2(f_2)^v\right\}grad(f_{1}^h)
\end{array}
\end{equation}
\end{lemma}
\section{Dualistic structure with respect to $G_{_{f_1f_2}}$.}
\begin{proposition}
  Let $(\Tilde{g}_{_{f_1f_2}}, \nabla, \nabla^{^{*}})$ be a dualistic structure on $M_{_1}\times M_{_2}$. Then there exists an
affine connections $\nabla{\hskip -0.2cm^{^{^{i}}}}$, $\nabla^{^{*}}{\hskip -0.36cm^{^{^{i}}}}~$ on $M_{_i}$, such that
$(g_{_{_i}},\nabla{\hskip -.2cm^{^{^{i}}}},\nabla^{^{*}}{\hskip -.36cm^{^{^{i}}}}~~~)$ is a dualistic structure on
$M_{_i}$ $(i=1,2)$.
\end{proposition}
\begin{proof}
Taking the affine connections on $M_{_i}$, $(i=1,2)$.
$$
\left\{
  \begin{array}{llll}
  (\nabla{\hskip -0.2cm^{^{^{1}}}}_{X_1}Y_1)\circ \pi_1=d\pi_1(\nabla_{X_1^h}Y_1^h)
+c\frac{f_1^h}{f_{2}^v}(\nabla_{X_1^h}Y_1^h)(f_{2}^v)(gradf_1)\circ \pi_1, & \forall~X_1,Y_1\in \Gamma(TM_{_1})\\
(\nabla^{^{*}}{\hskip -.4cm^{^{^{1}}}}_{X_1}Y_1)\circ \pi_1=d\pi_1(\nabla^*_{X_1^h}Y_1^h)
+c\frac{f_1^h}{f_{2}^v}(\nabla_{X_1^h}^*Y_1^h)(f_{2}^v)(gradf_1)\circ \pi_1,&\\
&&\\
(\nabla{\hskip -.2cm^{^{^{2}}}}_{X_2}Y_2)\circ \pi_2=d\pi_2(\nabla_{X_2^v}Y_2^v)
+c\frac{f_2^v}{f_{1}^h}(\nabla_{X_2^v}Y_2^v)(f_{1}^h)(gradf_2)\circ \pi_2, & \forall~X_2,Y_2\in \Gamma(TM_{_2})\\
(\nabla^{^{*}}{\hskip -.34cm^{^{^{2}}}}_{X_2}Y_2)\circ \pi_2=d\pi_2(\nabla^*_{X_2^v}Y_2^v)
+c\frac{f_2^v}{f_{1}^h}(\nabla_{X_2^v}^*Y_2^v)(f_{1}^h)(gradf_2)\circ \pi_2.&\\
  \end{array}
\right.
$$
Therfore, we have for all $X_i,Y_i,Z_i\in\Gamma(TM_{_i})$ $(i=1,2)$.
   \begin{equation}\label{horizdual}
    X_i^I(G_{_{f_1f_2}}(Y_i^I,Z_i^I))=G_{_{f_1f_2}}(\nabla_{X_i^I}Y_i^I,Z_i^I)+G_{_{f_1f_2}}(Y_i^I,\nabla^{^{*}}_{X_i^I}Z_i^I).
   \end{equation}
Since, $d\pi{\hskip -0.1cm{_{_{_{_{3-i}}}}}}\!\!\!\!\!(Z_i^I)=0$, $X_i^I(f_{3-i}^{J})=0$
 and for any $X\in\Gamma(TM_1\times M_2)$,
$$
g_{f_1f_2}(X,Z_i^I)=(f_{3-i}^J)^2g_i^{\pi_i}(d\pi_i(X),Z_i\circ\pi_i)+cf_1^hf_2^vX(f_{3-i}^J)Z_i(f_i)^I,
$$
then the equation (\ref{horizdual}) is aquivalent to
$$
(f_{3-i}^J)^{2}(X_i(g_{i}(Y_i,Z_i)))^I=(f_{3-i}^J)^{2}\big \{g_{i}(\nabla{\hskip -0.15cm^{^{^{i}}}}_{X_i}Y_i,Z_i)
+g_{i}(Y_i,\nabla^*{\hskip -0.3cm^{^{^{i}}}}_{X_i}Z_i)\}^I.
$$
Where $(i,I), (3-i,J)\in\{(1,h),(2,v)\}$.
Hence, the pair of affine connections $\nabla{\hskip -0.15cm^{^{^{i}}}}$ and $\nabla^{^{*}}{\hskip -0.3cm^{^{^{i}}}}~$  ~are
conjugate with respect to $g_{i}$.
\end{proof}
\begin{proposition}
Let $(g_{i},\nabla{\hskip -0.15cm^{^{^{i}}}},\nabla^{^{*}}{\hskip -0.3cm^{^{^{i}}}}~)$ be a dualistic structure on
$M_{i}$ $(i=1,2)$. Then there exists a dualistic structure on $M_{_1}\times M_{_2}$ with respect to $G_{f_1f_2}$.
\end{proposition}
\begin{proof}
Let $\nabla$ and $\nabla^*$ be the connections on $M_1\times M_2$ given by
\begin{equation}\label{full generalized}
\left\{
\begin{array}{lll}
d\pi_1(\nabla _XY)&\!\!\!\!=\nabla{\hskip-0.4cm^{^{^{^{\pi_{1}}}}}}_{X}d\pi_1(Y)+Y(\ln f{\!_{_{2}}}\!\!^{v})d\pi_1(X)
+X(\ln f{\!_{_{2}}}\!\!^{v})d\pi_1(Y)\\
&\!\!\!\!+\frac{1}{f_1^hf_2^v(1-c^2b_1^hb_2^v)}\big\{\frac{(f_1^h)^2}{f_2^v}
B{\hskip -0.08cm_{_{f_{_{2}}^v}}}(X,Y)-cb_2^vf_2^vB{\hskip -0.08cm_{_{f_{_{1}}^h}}}(X,Y)\\
&\hskip -1.2cm-cf_1^h(1-cb_2^v)\big[X(f_1^h)Y(f_2^v)+X(f_2^v)Y(f_1^h)\big]\big\}(gradf_1)\circ \pi_1,& \\
&\\
d\pi_2(\nabla _XY)&\!\!\!\!=\nabla{\hskip-0.4cm^{^{^{^{\pi_{2}}}}}}_{X}d\pi_2(Y)+Y(\ln f{\!_{_{1}}}\!\!^{h})d\pi_2(X)
+X(\ln f{\!_{_{1}}}\!\!^{h})d\pi_2(Y)\\
&\!\!\!\!+\frac{1}{f_1^hf_2^v(1-c^2b_1^hb_2^v)}\big\{\frac{(f_2^v)^2}{f_1^h}
B{\hskip -0.08cm_{_{f\!_{_{1}}^h}}}(X,Y)-cb_1^hf_1^hB{\hskip -0.08cm_{_{f\!_{_{2}}^v}}}(X,Y)\\
&\hskip -1.2cm-cf_2^v(1-cb_1^h)\big[X(f_1^h)Y(f_2^v)+X(f_2^v)Y(f_1^h)\big]\big\}(gradf_2)\circ \pi_2,& \\
&\\
d\pi_1(\nabla^* _XY)&\!\!\!\!= \nabla{\hskip-0.4cm^{^{^{^{\pi_1}}}}}_{X}{\!\!\!^*}d\pi_1(Y)+Y(\ln f{\!_{_{2}}}\!\!^{v})d\pi_1(X)
+X(\ln f{\!_{_{2}}}\!\!^{v})d\pi_1(Y)\\
&\!\!\!\!+\frac{1}{f_1^hf_2^v(1-c^2b_1^hb_2^v)}\big\{\frac{(f_1^h)^2}{f_2^v}
B^*{\hskip -0.2cm_{_{f_{_{2}}^v}}}(X,Y)-cb_2^vf_2^v
B^*{\hskip -0.2cm_{_{f_{_{1}}^h}}}(X,Y)\\
&\hskip -1.2cm-cf_1^h(1-cb_2^v)\big[X(f_1^h)Y(f_2^v)+X(f_2^v)Y(f_1^h)\big]\big\}(gradf_1)\circ \pi_1,&\\
&\\
d\pi_2(\nabla^* _XY)&\!\!\!\!= \nabla{\hskip-0.4cm^{^{^{^{\pi_2}}}}}_{X}{\!\!\!^*}d\pi_2(Y)+Y(\ln f{\!_{_{1}}}\!\!^{h})d\pi_2(X)
+X(\ln f{\!_{_{1}}}\!\!^{h})d\pi_2(Y)\\
&\!\!\!\!+\frac{1}{f_1^hf_2^v(1-c^2b_1^hb_2^v)}\big\{\frac{(f_2^v)^2}{f_1^h}
B^*{\hskip -0.2cm_{_{f_{_{1}}^h}}}(X,Y)-cb_1^hf_1^h
B^*{\hskip -0.2cm_{_{f_{_{2}}^v}}}(X,Y)\\
&\hskip -1.2cm-cf_2^v(1-cb_1^h)\big[X(f_1^h)Y(f_2^v)+X(f_2^v)Y(f_1^h)\big]\big\}(gradf_2)\circ \pi_2,&
\end{array}
\right.
\end{equation}

for any $X,Y\in \Gamma(TM_1\times M_2)$. Where  $B{\hskip -0.1cm_{_{f\!_{_{i}}^I}}}$
and $B^*{\hskip -0.3cm_{_{f\!_{_{i}}^I}}}$ $(i=1,2)$
the $(0,2)$ tensors fields of $f_{_i}^I$ given respectively by
$$
\begin{array}{ccl}
    B{\hskip -0.1cm_{_{f\!_{_{i}}^I}}}(X,Y)&=&cf_{_i}^I\left\{X(Y(f_i^I))-g_{_{_{i}}}{\hskip -0.1cm^{\pi_i}}\big(\nabla{\hskip -0.3cm^{^{^{\pi_i}}}}_{X}d\pi_i(Y),(gradf_i)\circ\pi_i\big)\right\}  \\
   &+&cX(f_{_i}^I)Y(f_{_i}^I)-\frac{1}{f_{j}^J}g_{_{_{i}}}{\hskip -0.1cm^{\pi_i}}\big(d\pi_i(X),d\pi_i(Y)\big),
  \end{array}
$$
and
$$
\begin{array}{ccl}
    B^*{\hskip -0.3cm_{_{f\!_{_{i}}^I}}}(X,Y)&=&cf_{_i}^I\left\{X(Y(f_i^I))-g_{_{_{i}}}{\hskip -0.1cm^{\pi_i}}\big(\nabla^*{\hskip -0.5cm^{^{^{\pi_i}}}}_{X}d\pi_i(Y),(gradf_i)\circ\pi_i\big)\right\}  \\
   &+&cX(f_{_i}^I)Y(f_{_i}^I)-\frac{1}{f_{j}^J}g_{_{_{i}}}{\hskip -0.1cm^{\pi_i}}\big(d\pi_i(X),d\pi_i(Y)\big),
  \end{array}
$$
$j=i-3$ and $(i,I), (j,J)\in \{(1,h),(2,v)\}$.
\\
Or, for any $X_i,Y_i\in \Gamma(TM_i)$ $(i=1,2)$
\begin{equation}  \label{dual generalized}
\left\{
  \begin{array}{lll}
    \nabla _{X_1^h}Y_1^{h}=(\nabla{\hskip -0.25cm^{^{^{1}}}}_{X_1}Y_1)^h
+f_{_{2}}^vB_{f_1}(X_1,Y_1)^hgrad(f_{_{2}}^v); &  \\
 \nabla _{X_2^v}Y_2^{vh}=(\nabla{\hskip -0.2cm^{^{^{2}}}}_{X_2}Y_2)^v
+f_{_{1}}^hB_{f_2}(X_2,Y_2)^vgrad(f_{_{1}}^h); &  \\
&\\
 \nabla^* _{X_1^h}Y_1^{h}=(\nabla^*{\hskip -0.38cm^{^{^{1}}}}_{X_1}Y_1)^h
+f_{_{2}}^vB_{f_1}^*(X_1,Y_1)^hgrad(f_{_{2}}^v); &  \\

 \nabla^* _{X_2^v}Y_2^{v}=(\nabla^*{\hskip -0.38cm^{^{^{2}}}}_{X_2}Y_2)^v
+f_{_{1}}^hB_{f_2}^*(X_2,Y_2)^vgrad(f_{_{1}}^h); &  \\
&\\
\nabla_{X_1^h}Y_{2}^v=\nabla^*_{X_1^h}Y_{2}^v=-cX_1(f_1)^hY_2(f_2)^v\big\{f_{_2}^vgrad(f_{_1}^h)+f_{_1}^hgrad(f_{_2}^v)\}&\\
\hskip 3.4cm+\big(Y_2(\ln f_{_{_{2}}})\big)^vX_1^h+\big(X_1(\ln f_2)\big)^hY_2^v\\

\nabla _{Y_2^v}X_{1}^h=\nabla^* _{Y_2^v}X_{1}^h=\nabla _{X_1^h}Y_{2}^v.

\end{array}
\right.
\end{equation}
Where $B{\hskip -0.08cm_{_{f\!_{_{i}}}}}$
and $B^*{\hskip -0.25cm_{_{f\!_{_{i}}}}}$ $(i=1,2)$
the $(0,2)$ tensors fields of $f_{_i}$ given respectively by
$$
B{\hskip -0.08cm_{_{f\!_{_{i}}}}}(X_i,Y_i)=cf_{_i}\left\{X_i(Y_i(f_i))
-\nabla{\hskip -0.2cm^{^{^{i}}}}_{X_i}Y_i(f_i)\right\}+cX_i(f_{_i})Y_i(f_{_i})-g_i(X_i,Y_i),
$$
and
$$
B^*{\hskip -0.2cm_{_{f\!_{_{i}}}}}(X_i,Y_i)=cf_{_i}\left\{X_i(Y_i(f_i))
-\nabla^*{\hskip -0.35cm^{^{^{i}}}}_{X_i}Y_i(f_i)\right\}+cX_i(f_{_i})Y_i(f_{_i})-g_i(X_i,Y_i),
$$
Let us assume that $(g_{i},\nabla{\hskip -0.17cm^{^{^{i}}}},\nabla^*{\hskip -0.33cm^{^{^{i}}}}~)$ is
a dualistic structures on $M_{_i}$, $i=1,2$. Let $A$ be a tensor field of type $(0,3)$ defined
for any $X,Y, Z\in \Gamma(TM_1\times M_2)$ by
$$
A(X,Y,Z)=X(G_{f_1f_2}(Y,Z))-G_{f_1f_2}(\nabla_XY,Z)-G_{f_1f_2}(Y,\nabla^*_XZ),
$$
if  $X_i,Y_i, Z_i\in \Gamma(TM_i)$, $i=1,2$, then we have
$$
 X_i^I(G_{_{f_1f_2}}(Y_i^I,Z_i^I))=X_i^I((f_{3-i}^J)^2g_i(X_i,Y_i)^I).
$$
Since $d\pi_{3-i}(X_i^I)=0$, it follows that $d\pi_{3-i}(X_i^I)(f_{3-i}=X_i^I(f_{3-i}^J)=0$, and hence
$$
 X_i^I(G_{_{f_1f_2}}(Y_i^I,Z_i^I))=(f_{3-i}^J)^2(X(g_i(Y_i,Z_i)))^I,
$$
as $(g_{i},\nabla{\hskip -0.17cm^{^{^{i}}}},\nabla^*{\hskip -0.33cm^{^{^{i}}}}~)$
 is dualistic structure, we have thus
$$
 X_i^I(G_{_{f_1f_2}}(Y_i^I,Z_i^I))=(f_{3-i}^J)^2\{g_i(\nabla{\hskip -0.17cm^{^{^{i}}}}_{X_i}Y_i,Z_i)^I
+g_i(Y_i,\nabla^*{\hskip -0.33cm^{^{^{i}}}}_{X_i}Z_i)^I\}, $$
from Equations (\ref{equivalent generalized warped}), (\ref{dual generalized}),
then it's easily seen that the following equation holds
$$
A(X_i^I,Y_i^I, Z_i^I)=0
$$
In the different lifts $(i\neq j)$, we have
$$
X_{i}^I(G_{_{f_1f_2}}(Y_i^I,Z_j^J))=cf_j^J(Z_j(f_j))^JX_i((f_i(Y(f_i))))^I,
$$
$$
G_{_{f_1f_2}}(\nabla_{X_{i}^J}Y_i^I,Z_j^J)=f_j^J\left\{cf_iX_i(Y_i(f_i))+cX_i(f_i)Y_i(f_i)-g_i(X_i,Y_i)
\right\}^IZ_j(f_j)^J,
$$
and
$$
G_{_{f_1f_2}}(\nabla_{X_{i}^I}^*Z_j^J,Y_i^I)=f_j^Jg_i(X_i,Y_i)^IZ_j(f_j)^J.
$$
We add these equations and obtain
$$
A(X_{i}^I,Y_i^I, Z_j^J)=0
$$
Hence the same applies for $A(X_j^J,Y_i^I, Z_{i}^I)=A(X_{i}^I,Y_j^J, Z_i^I)=0$.   \\
This proves that $\nabla^*$ is conjugate to $\nabla$ with respect to $G_{_{f_1f_2}}.$

\end{proof}

We recall that the connection $\nabla$ on $M_{_1}\times M_{_2}$ induced by $\nabla{\hskip -0.19cm^{^{^{1}}}}$ and
 $\nabla{\hskip -0.19cm^{^{^{2}}}}$ on $M_{_1}$ and $M_{_2}$ respectively, is given by Equation
(\ref{dual generalized}).

\begin{proposition}
$(M_{_1},\nabla{\hskip -0.19cm^{^{^{1}}}},g_1)$ and $(M_{_2},\nabla{\hskip -0.19cm^{^{^{2}}}},g_2)$ are
statistical manifolds if and only if $(M_{_1}\times M_{_2},G_{_{f_1f_2}},\nabla)$ is a statistical manifold.
\end{proposition}
\begin{proof}
    Let us assume that $(M_{_i},\nabla{\hskip -0.17cm^{^{^{i}}}},g_{_{_i}})$ $(i=1,2)$ is statistical manifold.\\
    Firstly, we show that $\nabla$ is torsion-free. Indeed; by Equation (\ref{full generalized}), we have
for any $X,Y\in\Gamma(TM_1\times M_2)$
$$
        d\pi_i(T(X,Y)) =  \nabla{\hskip-0.4cm^{^{^{\pi_i}}}}_{X}d\pi_i(Y)
-\nabla{\hskip-0.4cm^{^{^{\pi_i}}}}_{Y}d\pi_i(X)-d\pi_i([X,Y])
$$
Since for $i=1,2$, $\nabla{\hskip -0.17cm^{^{^{i}}}}$ is torsion-free, then
$$
\nabla{\hskip-0.4cm^{^{^{\pi_i}}}}_{X}d\pi_i(Y)
-\nabla{\hskip-0.4cm^{^{^{\pi_i}}}}_{Y}d\pi_i(X)=d\pi_i([X,Y])
$$
   Therefore, from Remark \ref{rem dpi}, the connection $\nabla$ is torsion-free.\\
Secondly, we show that $\nabla G_{f_1,f_2}$ is symmetric. In fact; for $i=1,2$,
$$
(\nabla G_{_{f_1f_2}})(X_i^I,Y_i^I,Z_i^J)=X_i^I(G_{_{f_1f_2}}(Y_i^I,Z_i^I))-G_{_{f_1f_2}}(\nabla_{X_i^I}Y_i^I,Z_i^I)
-G_{_{f_1f_2}}(Y_i^I,\nabla_{X_i^I}Z_i^I)
$$
by Equations (\ref{equivalent generalized warped}), (\ref{dual generalized}) and since
$(\nabla{\hskip -0.17cm^{^{^{i}}}}g_i)$, $i=1,2$, is symmetric, we have
      $$\begin{array}{lll}
          (\nabla G_{_{f_1f_2}})(X_i^I,Y_I^I,Z_i^I) & = & (f_{3-i}^J)^2((\nabla{\hskip -0.17cm^{^{^{i}}}}g_i)(X_i,Y_i,Z_i))^I \\
           & = &(f_{3-i}^J)^2((\nabla{\hskip -0.17cm^{^{^{i}}}}g_i)(Y_i,X_i,Z_i))^h  \\
           & = &  (\nabla G_{_{f_1f_2}})(Y_i^I,X_I^I,Z_i^I).
        \end{array}
      $$
In the different lifts, we have
$$
(\nabla G_{_{f_1f_2}})(X_i^I,Y_i^I, Z_{3-i}^J)=(\nabla G_{_{f_1f_2}})(X_{3-i}^J,Y_i^I, Z_i^I)
=(\nabla G_{_{f_1f_2}})(X_i^I,Y_{3-i}^I, Z_{i}^I)=0,
$$
Therefore, $(\nabla G_{_{f_1f_2}})$ is symmetric.
Thus $(M_{_1}\times M_{_2},g_{_{f_1f_2}},\nabla)$ is a statistical manifold.\\

Conversely, if $(M_{_1}\times M_{_2},G_{_{f_1f_2}},\nabla)$ is statistical manifold,
then $(\nabla G_{_{f_1f_2}})$ is symmetric and $\nabla$ is torsion-free, particularly,
when $X_i, Y_i, Z_i\in \Gamma(TM_i)$, we have
$$
\left\{
    \begin{array}{rll}
     (\nabla G_{_{f_1f_2}})(X_i^I,Y_I^I,Z_i^I)&=(\nabla G_{_{f_1f_2}})(Y_i^I,X_I^I,Z_i^I),& \\
       & &\forall~i=1,2,\\
      T(X_i^I,Y_i^I)&=0.&
    \end{array}
  \right.
$$
Then, by Equations (\ref{equivalent generalized warped}) and (\ref{dual generalized}), we
obtained, for $i=1,2$, $\nabla{\hskip -0.17cm^{^{^{i}}}}g_i$, is symmetric and $\nabla{\hskip -0.17cm^{^{^{i}}}}$,
is torsion-free. Therefore, $(M_i,\nabla{\hskip -0.17cm^{^{^{i}}}},g_i)$, $i=1,2$, is statistical manifold.
\end{proof}
\section{Dualistic structure with respect to $\tilde{g}_{_{f_1f_2}}$.}
Let $c$ be an arbitrary real number and let $g_i$, $(i=1,2)$ be a Riemannian metric
tensors on $M_i$. Given a smooth positive function $f_i$ on $M_i$,
we define a metric tensor field on $M_1\times M_2$ by
\begin{equation}\label{Other metric}
        \tilde{g}_{_{f_1,f_2}}=\pi_1^*g_{_{_1}}+(f_1^h)^2\pi_2^*g_{_{_2}}
        +\frac{c^2}{2}(f_2^v)^2 df_1^h\odot df_1^h.
\end{equation}
Where $\pi_i$, $(i=1,2)$ is the projection of $M_{_1}\times M_{_2}$ onto $M_{_i}$ (see \cite{Nass2}).\\
For all $X,Y\in\Gamma(TM_1\times M_2)$, we have
$$
\begin{array}{rl}
 \tilde{g}_{_{f_1,f_2}}(X,Y)&=g_{_{_1}}^{\pi_1}(d\pi_1(X),d\pi_1(Y))
 +(f_1^h)^2g_{_{_2}}^{\pi_2}(d\pi_2(X),d\pi_2(Y))+(cf_2^v)^2X(f_1^h)Y(f_1^h)).
\end{array}
$$
It is the unique tensor fields such that for any $X_i,Y_i\in\Gamma(TM_i)$, $(i=1,2)$
\begin{equation} \label{Other equivalent generalized warped}
 \left\{
  \begin{array}{ll}
\tilde{g}_{_{f_1f_2}}(X_1^h,Y_1^h)=g_1(X_1,Y_1)^h+(cf_2^v)^2X_1(f_1)Y_1(f_1)^h, & \\
\tilde{g}_{_{f_1f_2}}(X_1^h,Y_2^v)=\tilde{g}_{_{f_1f_2}}(Y_2^v,X_1^h)=0,&\\
\tilde{g}_{_{f_1f_2}}(X_2^v,Y_2^v)=(f_{1}^{h})^2g_2(X_2,Y_2)^v.
  \end{array}
\right.
\end{equation}
\begin{proposition}
  Let $(\tilde{g}_{_{f_1f_2}}, \nabla, \nabla^*)$ be a dualistic structure on $M_{_1}\times M_{_2}$. Then there exists an
affine connections $\nabla{\hskip -0.15cm^{^{^{i}}}}$, $\nabla^*{\hskip -0.3cm^{^{^{i}}}}~$ on $M_{_i}$, such that
$(g_{_{_i}},\nabla{\hskip -.15cm^{^{^{i}}}},\nabla^*{\hskip -.3cm^{^{^{i}}}}~)$ is a dualistic structure on
$M_{_i}$ $(i=1,2)$.
\end{proposition}
\begin{proof}
Taking the affine connections on $M_{_i}$, $(i=1,2)$.
\begin{equation}\label{connection1,*}
    \left\{
       \begin{array}{ll}
        (\nabla{\hskip -0.2cm^{^{^{1}}}}_{X_1}Y_1)\circ \pi_1=d\pi_1(\nabla_{X_1^h}Y_1^h)
+(cf_2^v)^2H^{f_1^h}(X_1^h,Y_1^h)(gradf_1)\circ \pi_1, & \\
(\nabla^{^{*}}{\hskip -.4cm^{^{^{1}}}}_{X_1}Y_1)\circ \pi_1=d\pi_1(\nabla^*_{X_1^h}Y_1^h)
+(cf_2^v)^2H^{*f_1^h}(X_1^h,Y_1^h)(gradf_1)\circ \pi_1,&\\
       \end{array}
     \right.
\end{equation}
\begin{equation}\label{connection2,*}
\left\{
  \begin{array}{llll}
(\nabla{\hskip -.2cm^{^{^{2}}}}_{X_2}Y_2)\circ \pi_2=\frac{1}{(f_1^h)^2}d\pi_2(\nabla_{X_2^v}Y_2^v)\\
(\nabla^{^{*}}{\hskip -.36cm^{^{^{2}}}}_{X_2}Y_2)\circ \pi_2=\frac{1}{(f_1^h)^2}d\pi_2(\nabla^*_{X_2^v}Y_2^v).&\\
  \end{array}
\right.
\end{equation}
Therfore, we have for all $X_i,Y_i,Z_i\in\Gamma(TM_{_i})$ $(i=1,2)$.
   \begin{equation}\label{horizdual}
    X_i^I(\Tilde{g}_{_{f_1f_2}}(Y_i^I,Z_i^I))=\tilde{g}_{_{f_1f_2}}(\nabla_{X_i^I}Y_i^I,Z_i^I)+\tilde{g}_{_{f_1f_2}}(Y_i^I,\nabla^*_{X_i^I}Z_i^I).
   \end{equation}
Since, $d\pi{\hskip -0.1cm{_{_{_{_{3-i}}}}}}\!\!\!\!\!(Z_i^I)=0$, $X_i^I(f_{3-i}^{J})=0$
 and for any $X\in\Gamma(TM_1\times M_2)$,
$$
\tilde{g}_{f_1f_2}(X,Z_i^I)=
\left\{
  \begin{array}{ll}
  g_1^{\pi_1}(d\pi_1(X),Z_1\circ\pi_1)+(cf_2^v)^2X(f_1^h)Z_i(f_1)^h, & if (i,I)=(1,h) \\
   ( f_{1}^h)^2g_2^{\pi_2}(d\pi_2(X),Z_2\circ\pi_2), & (i,I)=(2,v)
  \end{array}
\right.
$$
Substituting from Equations (\ref{connection1,*}) and (\ref{connection2,*}) into
Formula (\ref{horizdual}) gives
$$
\left\{
  \begin{array}{ll}
    \left(X_1(g_1(Y_1,Z_1))\right)^h=g_1^{\pi_1}(\nabla{\hskip -.2cm^{^{^{1}}}}_{X_1}Y_1,Z_1\circ\pi_1)
+g_1^{\pi_1}(\nabla^{^{*}}{\hskip -.38cm^{^{^{1}}}}_{X_1}Z_1,Y_1\circ\pi_1), &  \\
    (f_1^h)^2 \left(X_2(g_2(Y_2,Z_2))\right)^v=(f_1^h)^2\left\{g_2^{\pi_2}(\nabla{\hskip -.2cm^{^{^{2}}}}_{X_2}Y_2,Z_2\circ\pi_2)
+g_2^{\pi_2}(\nabla^{^{*}}{\hskip -.38cm^{^{^{2}}}}_{X_2}Z_2,Y_2\circ\pi_2)\right\}, &
  \end{array}
\right.
$$
Hence, the pair of affine connections $\nabla{\hskip -0.15cm^{^{^{i}}}}$ and $\nabla^{^{*}}{\hskip -0.35cm^{^{^{i}}}}~$  ~are
conjugate with respect to $g_{i}$.
\end{proof}
\begin{proposition}
Let $(g_{i},\nabla{\hskip -0.15cm^{^{^{i}}}},\nabla^{^{*}}{\hskip -0.35cm^{^{^{i}}}}~)$ be a dualistic structure on
$M_{i}$ $(i=1,2)$. Then there exists a dualistic structure on $M_{_1}\times M_{_2}$ with respect to $\tilde{g}_{f_1f_2}$.
\end{proposition}
\begin{proof}
Let $\nabla$ and $\nabla^*$ be the connections on $M_1\times M_2$ given by
\begin{equation}  \label{Other dual generalized}
\left\{
  \begin{array}{rll}
    \nabla _{X_1^h}Y_1^{h}&=&(\nabla{\hskip -0.25cm^{^{^{1}}}}_{X_1}Y_1)^h
+\frac{(cf_2^v)^2 H^{f_1}(X_1,Y_1)^h}{1+(cf_2^v)^2b_1^h}(gradf_1)^h   \\
&&\\
& -&c^2f_2^v(X_1(\ln f_1)Y_1(\ln f_1))^h(gradf_2)^v,\\
&&\\
 \nabla _{X_2^v}Y_2^{vh}&=&(\nabla{\hskip -0.2cm^{^{^{2}}}}_{X_2}Y_2)^v
 -\frac{f_1^hg_2(X_2,Y_2)^v}{1+(cf_2^v)^2b_1^h}(gradf_1)^h,  \\
&&\\

 \nabla^* _{X_1^h}Y_1^{h}&=&(\nabla^*{\hskip -0.38cm^{^{^{1}}}}_{X_1}Y_1)^h
+\frac{(cf_2^v)^2 H{^{*}}^{f_1}(X_1,Y_1)^h}{1+(cf_2^v)^2b_1^h}(gradf_1)^h   \\
&&\\
& -&c^2f_2^v(X_1(\ln f_1)Y_1(\ln f_1))^h(gradf_2)^v,  \\
&&\\
 \nabla^* _{X_2^v}Y_2^{v}&=&(\nabla^*{\hskip -0.38cm^{^{^{2}}}}_{X_2}Y_2)^v
-\frac{f_1^hg_2(X_2,Y_2)^v}{1+(cf_2^v)^2b_1^h}(gradf_1)^h,   \\
&\\
\nabla_{X_1^h}Y_{2}^v&=&\nabla^*_{X_1^h}Y_{2}^v=\frac{c^2f_2^vY_2(f_2)^vX_1(f_1)^h}{1+(cf_2^v)^2b_1^h}(gradf_1)^h
 +\big(X_1(\ln f_{_{1}})\big)^hY_2^v,\\
&&\\

\nabla _{Y_2^v}X_{1}^h&=&\nabla^* _{Y_2^v}X_{1}^h=\nabla _{X_1^h}Y_{2}^v.

\end{array}
\right.
\end{equation}
for any $X_i,Y_i\in \Gamma(TM_i)$ $(i=1,2)$. Where $H^{f_1}$ and $H{^{*}}^{f_1}$ are the
Hessian of $f_{_1}$ with respect to $\nabla{\hskip -0.2cm^{^{^{1}}}}$ and
$\nabla^{^{*}}{\hskip -0.36cm^{^{^{1}}}}$ respectively.\\
Let us assume that $(g_{i},\nabla{\hskip -0.17cm^{^{^{i}}}},\nabla^{^{*}}{\hskip -0.33cm^{^{^{i}}}}~)$ is
a dualistic structures on $M_{_i}$, $i=1,2$. Let $A$ be a tensor field of type $(0,3)$ defined
for any $X,Y, Z\in \Gamma(TM_1\times M_2)$ by
$$
A(X,Y,Z)=X(\tilde{g}_{f_1f_2}(Y,Z))-\tilde{g}_{f_1f_2}(\nabla_XY,Z)-\tilde{g}_{f_1f_2}(Y,\nabla^{^{*}}_XZ),
$$
Since $d\pi_{3-i}(X_i^I)=0$, it follows that
$$
X_i^I(f_{3-i}^J)=d\pi_{3-i}(X_i^I)(f_{3-i})=0, ~~\hskip 0.4cm
\forall (i,I),(j,J)\in\{(i,h),(2,v)\},
$$
and hence, for all $X_i,Y_i, Z_i\in \Gamma(TM_i)$ $(i=1,2)$, we have
$$
\left\{
  \begin{array}{ll}
    X_1^h\left(\tilde{g}_{f_1f_2}(Y_1^h,Z_1^h)\right)=\left(X_1(g_1(Y_1,Z_1))\right)^h
+(cf_2^v)^2\left\{Y_1(f_1)X_1(Z_1(f_1))+Z_1(f_1)X_1(Y_1(f_1))\right\}^h, & \\
     X_2^v\left(\tilde{g}_{f_1f_2}(Y_2^v,Z_2^v)\right)=(cf_2^v)^2\left(X_2(g_2(Y_2,Z_2))\right)^h. &
  \end{array}
\right.
$$
as $(g_{i},\nabla{\hskip -0.17cm^{^{^{i}}}},\nabla^*{\hskip -0.33cm^{^{^{i}}}}~)$
 is dualistic structure and from Equations (\ref{Other equivalent generalized warped}), (\ref{Other dual generalized}),
then it's easily seen that the following equation holds
$$
A(X_i^I,Y_i^I, Z_i^I)=0, \hskip 0.5cm \forall (i,I),(j,J)\in\{(i,h),(2,v)\}.
$$
In the different lifts $(i\neq j)$, we have
$$
X_{i}^I(\tilde{g}_{_{f_1f_2}}(Y_i^I,Z_j^J))=0,
$$
$$
\left\{
  \begin{array}{ll}
   \tilde{g}_{_{f_1f_2}}(\nabla_{X_{1}^h}Y_1^h,Z_2^v)=-c^2f_2^vX_1(f_1)^hY_1(f_1)^hZ_2(f_2)^v, &  \\
    \tilde{g}_{_{f_1f_2}}(\nabla_{X_{2}^v}Y_2^v,Z_1^h)=-f_1^hg_2(X_2,Y_2)^vZ_1(f_1)^h, &
  \end{array}
\right.
$$
and
$$
\left\{
  \begin{array}{ll}
    \tilde{g}_{_{f_1f_2}}(Y_1^h,\nabla^{^{*}}_{X_{1}^h}Z_2^v)=c^2f_2^vX_1(f_1)^hY_1(f_1)^hZ_2(f_2)^v, &  \\
     \tilde{g}_{_{f_1f_2}}(Y_2^v,\nabla^{^{*}}_{X_{2}^v}Z_1^h)=f_1^hg_2(X_2,Y_2)^vZ_1(f_1)^h, &
  \end{array}
\right.
$$
We add these equations and obtain
$$
A(X_{i}^I,Y_i^I, Z_j^J)=0,  \hskip 0.5cm \forall (i,I),(j,J)\in\{(i,h),(2,v)\}.
$$
Hence the same applies for $A(X_j^J,Y_i^I, Z_{i}^I)=A(X_{i}^I,Y_j^J, Z_i^I)=0$.   \\
This proves that $\nabla^{^{*}}$ is conjugate to $\nabla$ with respect to $\tilde{g}_{_{f_1f_2}}.$

\end{proof}

We recall that the connection $\nabla$ on $M_{_1}\times M_{_2}$ induced by $\nabla{\hskip -0.19cm^{^{^{1}}}}$ and
 $\nabla{\hskip -0.19cm^{^{^{2}}}}$ on $M_{_1}$ and $M_{_2}$ respectively, is given by Equation
(\ref{Other dual generalized}).

\begin{proposition}
$(M_{_1},\nabla{\hskip -0.19cm^{^{^{1}}}},g_1)$ and $(M_{_2},\nabla{\hskip -0.19cm^{^{^{2}}}},g_2)$ are
statistical manifolds if and only if $(M_{_1}\times M_{_2},\tilde{g}_{_{f_1f_2}},\nabla)$ is a statistical manifold.
\end{proposition}
\begin{proof}
    Let us assume that $(M_{_i},\nabla{\hskip -0.17cm^{^{^{i}}}},g_{_{_i}})$ $(i=1,2)$ is statistical manifold.\\
    Firstly, we show that $\nabla$ is torsion-free. Indeed; by Equation (\ref{Other dual generalized}), we have
for any $X,Y\in\Gamma(TM_1\times M_2)$
$$
        d\pi_i(T(X,Y)) =  \nabla{\hskip-0.4cm^{^{^{\pi_i}}}}_{X}d\pi_i(Y)
-\nabla{\hskip-0.4cm^{^{^{\pi_i}}}}_{Y}d\pi_i(X)-d\pi_i([X,Y])
$$
Since for $i=1,2$, $\nabla{\hskip -0.17cm^{^{^{i}}}}$ is torsion-free, then
$$
\nabla{\hskip-0.4cm^{^{^{\pi_i}}}}_{X}d\pi_i(Y)
-\nabla{\hskip-0.4cm^{^{^{\pi_i}}}}_{Y}d\pi_i(X)=d\pi_i([X,Y])
$$
   Therefore, from Remark \ref{rem dpi}, the connection $\nabla$ is torsion-free.\\
Secondly, we show that $\nabla G_{f_1,f_2}$ is symmetric. In fact; for $(i,I)\in\{(i,h),(2,v)\}$,
$$
(\nabla \tilde{g}_{_{f_1f_2}})(X_i^I,Y_i^I,Z_i^I)=X_i^I(\tilde{g}_{_{f_1f_2}}(Y_i^I,Z_i^I))-\tilde{g}_{_{f_1f_2}}(\nabla_{X_i^I}Y_i^I,Z_i^I)
-\tilde{g}_{_{f_1f_2}}(Y_i^I,\nabla_{X_i^I}Z_i^I)
$$
by Equations (\ref{Other equivalent generalized warped}), (\ref{Other dual generalized}) and since
$(\nabla{\hskip -0.17cm^{^{^{i}}}}g_i)$, $i=1,2$, is symmetric, we have
$$
(\nabla \tilde{g}_{_{f_1f_2}})(X_i^I,Y_i^I,Z_i^I)=(\nabla \tilde{g}_{_{f_1f_2}})(Y_i^I,X_i^I,Z_i^I).
$$
In the different lifts, for all $(i,I),(j,J)\in\{(i,h),(2,v)\}$, we have
$$
(\nabla \tilde{g}_{_{f_1f_2}})(X_i^I,Y_i^I, Z_{3-i}^J)=(\nabla \tilde{g}_{_{f_1f_2}})(X_{3-i}^J,Y_i^I, Z_i^I)
=(\nabla \tilde{g}_{_{f_1f_2}})(X_i^I,Y_{3-i}^J, Z_{i}^I)=0.
$$
Therefore, $(\nabla \tilde{g}_{_{f_1f_2}})$ is symmetric.
Thus $(M_{_1}\times M_{_2},\tilde{g}_{_{f_1f_2}},\nabla)$ is a statistical manifold.\\

Conversely, if $(M_{_1}\times M_{_2},\tilde{g}_{_{f_1f_2}},\nabla)$ is statistical manifold,
then $(\nabla \tilde{g}_{_{f_1f_2}})$ is symmetric and $\nabla$ is torsion-free, particularly,
when $X_i, Y_i, Z_i\in \Gamma(TM_i)$, we have
$$
\left\{
    \begin{array}{rll}
     (\nabla \tilde{g}_{_{f_1f_2}})(X_i^I,Y_I^I,Z_i^I)&=(\nabla \tilde{g}_{_{f_1f_2}})(Y_i^I,X_I^I,Z_i^I),& \\
       & &\forall~i=1,2,\\
      T(X_i^I,Y_i^I)&=0.&
    \end{array}
  \right.
$$
Then, by Equations (\ref{Other equivalent generalized warped}) and (\ref{Other dual generalized}), we
obtained, for $i=1,2$, $\nabla{\hskip -0.17cm^{^{^{i}}}}g_i$, is symmetric and $\nabla{\hskip -0.17cm^{^{^{i}}}}$,
is torsion-free. Therefore, $(M_i,\nabla{\hskip -0.17cm^{^{^{i}}}},g_i)$, $i=1,2$, is statistical manifold.
\end{proof}
At first, note that $(M_1\times M_2, \tilde{g}_{_{f_1f_2}}, \nabla)$ is the statistical manifold
induced from $(M_1, g_1,\nabla{\hskip -.21cm^{^{^{1}}}})$ and $(M_2, g_2,\nabla{\hskip -.21cm^{^{^{2}}}})$.\\
Now, let $(M_{_1},\nabla{\hskip -0.21cm^{^{^{1}}}},g_1)$ and
$(M_{_2},\nabla{\hskip -0.21cm^{^{^{2}}}},g_2)$ be two statistical
manifolds and let $\mathcal{R}{\hskip -0.2cm^{^{^{1}}}}$,
$\mathcal{R}{\hskip -0.2cm^{^{^{2}}}}$ and $\mathcal{R}$ be the curvature
tensors with respect to $\nabla{\hskip -0.21cm^{^{^{1}}}}$,
$\nabla{\hskip -0.21cm^{^{^{2}}}}$ and $\nabla$ respectively.
\begin{proposition}\label{Other curvature generalized}
Let $(M_{_i},\nabla{\hskip -0.21cm^{^{^{i}}}},\nabla^{^{*}}{\hskip -0.38cm^{^{^{i}}}}
,g_{_{_i}})$, $(i=1,2)$ be a connected statistical manifold.
Assume that the gradient of $f_i$ is parallel with respect to
$\nabla{\hskip -0.21cm^{^{^{i}}}}$ and $\nabla^{^{*}}{\hskip -0.38cm^{^{^{1}}}}$ $(i=1,2)$.
Then for any ${X_i},{Y_i},{Z_i}\in\Gamma(TM_{_i})$
$(i=1,2)$ we have
$$
\begin{array}{ll}
1.~\mathcal{R}({X_1}^h,{Y_1}^h){Z_1}^h&\!\!\!\!=(\mathcal{R}^{^1}({X_1},{Y_1}){Z_1})^h,\\
&\\
2.~\mathcal{R}({X_2}^v,{Y_2}^v){Z_2}^v&\!\!\!\!=(\mathcal{R}^{^2}({X_2},{Y_2}){Z_2})^v
-\frac{b_1}{1+(cf_2^v)^2b_1}\left\{(X_2\wedge_{g_2} Y_2)Z_2\right\}^v\\
&\\
&\!\!\!\!+\frac{c^2f_1^hf_2^vb_1}{\left(1+(cf_2^v)^2b_1\right)^2}
\left\{\left((X_2\wedge_{g_2} Y_2)Z_2\right)(f_2)\right\}^v(gradf_1)^h,\\
&\\
3.~\mathcal{R}({X_1}^h,{Y_1}^h){Z_2}^v&\!\!\!\!=0,\\
&\\

4.~\mathcal{R}({X_1}^h,{Y_2}^v){Z_1}^h&\!\!\!\!= \frac{c^2X_1(\ln f_1)^hZ_1(\ln f_1)^h Y_2(f_2)^v}{1+(cf_2^v)^2b_1}(gradf_2)^v,\\
\end{array}
$$
where the wedge product $(X_2\wedge_{g_2} Y_2)Z_2=g_2(Y_2,Z_2)X_2-g_2(X_2,Z_2)Y_2$.
\end{proposition}
\begin{proof}
Long but straightforward calculations as in proof of the proposal (2), where it uses the fact that connections are compatible with the metric.
We obtained the same results as in (2), knowing we use only the connections are symmetrical.
\end{proof}
\begin{corollary}
Let $(M_{_i},\nabla{\hskip -0.21cm^{^{^{i}}}},\nabla^{^{*}}{\hskip -0.38cm^{^{^{1}}}}
,g_{_{_i}})$, $(i=1,2)$ be a connected statistical manifold. Assume that $f_1$ is a non-constant
positive function and $c\neq 0$.\\
If $(\nabla, \nabla^{^{*}}, \tilde{g}_{_{f_1f_2}})$ is a dually flat structure then
$(\nabla{\hskip -0.21cm^{^{^{1}}}},\nabla^{^{*}}{\hskip -0.38cm^{^{^{1}}}}, g_1~)$
 is also dually flat and  $(\nabla{\hskip -0.21cm^{^{^{2}}}},\nabla^{^{*}}
{\hskip -0.38cm^{^{^{2}}}}, g_2~)$ has a constant sectional curvature.
\end{corollary}
\begin{proof}
Let $(\nabla, \nabla^{^{*}},\tilde{g}_{_{f_1f_2}})$ be a dually flat structure.\\
By1. of Proposition \ref{Other curvature generalized}, for any $X_1,Y_1,Z_1\in\Gamma(TM_{_1})$, we have
$$
\mathcal{R}{\hskip -0.2cm^{^{^{1}}}}(X_1,Y_1)Z_1=0,
$$
From Equation (\ref{R et R*}), Since $(M_{_1},\nabla{\hskip -0.21cm^{^{^{1}}}},g_1)$
$(i=1,2)$ is a statistical manifold, we have
$$
\mathcal{R}^{^{*}}{\hskip -0.2cm^{^{^{1}}}}(X_1,Y_1)Z_1=0.
$$
Hence $(M_{_1},\nabla{\hskip -0.21cm^{^{^{1}}}},\nabla^{^{*}}{\hskip -0.38cm^{^{^{1}}}},g_1)$ is dually flat.\\
By 4. of Proposition \ref{Other curvature generalized}, for any $X_1,Z_1\in\Gamma(TM_{_1})$ and
$Y_2\in\Gamma(TM_{_2})$ , we have
$$
\frac{c^2X_1(\ln f_1)^hZ_1(\ln f_1)^h Y_2(f_2)^v}{1+(cf_2^v)^2b_1}(gradf_2)^v=0.
$$
So $f_2$ is a constant function since $f_1$ is non-constant function and
$M_2$ is assumed to be connected.\\
Moreover, By 2. of Proposition \ref{Other curvature generalized}, for any
$X_2,Y_2,Z_2\in\Gamma(TM_{_2})$, we have
$$
\mathcal{R}{\hskip -0.2cm^{^{^{2}}}}(X_2,Y_2)Z_2=\frac{b_1}{1+(cf_2^v)^2b_1}\left\{(X_2\wedge_{g_2} Y_2)Z_2\right\}^v,
$$
Since $b_1$ and $f_2$ are constants, it follows from the previous equality that $(\nabla{\hskip -0.21cm^{^{^{2}}}},\nabla^{^{*}}
{\hskip -0.38cm^{^{^{2}}}}, g_2~)$ has a constant sectional curvature $\frac{b_1}{1+(cf_2^v)^2b_1}$.
\end{proof}
\textbf{Acknowledgement:} A big part of this work was done at The Raphel Salem Laboratory of mathematics,
University of Rouen (France), Rafik Nasri and Djelloul Djebbouri would like to thank Raynaud de Fitte Paul,
Siman Raulot for very useful discussions and the Mathematic section for their hospitality.


\medskip
Received xxxx 20xx; revised xxxx 20xx.
\medskip

\end{document}